\documentclass[11pt]{article}
\usepackage{amssymb,amsfonts, mathrsfs,amsmath,graphicx,enumerate}
\usepackage{setspace}
\usepackage{paralist}
\usepackage{amsthm,calc,graphicx,iwona,float,wrapfig,caption}
\usepackage{microtype}
\usepackage[colorlinks=true,citecolor=black,linkcolor=black,urlcolor=blue]{hyperref}
\usepackage[noabbrev,capitalise]{cleveref}
\usepackage[lmargin=30mm,rmargin=30mm,bmargin=27mm,tmargin=27mm]{geometry}
\usepackage[numbers,sort&compress]{natbib}
\renewcommand{\baselinestretch}{1.2}
\usepackage[hang]{footmisc} 

\renewcommand{\thefootnote}{\fnsymbol{footnote}}	
\allowdisplaybreaks
\newcommand\DateFootnote{
\begingroup
\renewcommand\thefootnote{}
\footnote{
\today}
\setcounter{footnote}{0}
\vspace*{-3ex}
\endgroup}

\makeatletter
\renewcommand\section{\@startsection {section}{1}{\z@}%
                                   {-3ex \@plus -1ex \@minus -.2ex}%
                                   {2ex \@plus.2ex}%
                                   {\normalfont\large\bfseries}}
\renewcommand\subsection{\@startsection{subsection}{2}{\z@}%
                                     {-2.5ex\@plus -1ex \@minus -.2ex}%
                                     {1.5ex \@plus .2ex}%
                                     {\normalfont\normalsize\bfseries}}
\renewcommand\subsubsection{\@startsection{subsubsection}{3}{\z@}%
                                     {-2ex\@plus -1ex \@minus -.2ex}%
                                     {1ex \@plus .2ex}%
                                     {\normalfont\normalsize\bfseries}}
 \renewcommand\paragraph{\@startsection{paragraph}{4}{\z@}%
                                    {1.5ex \@plus.5ex \@minus.2ex}%
                                    {-1em}%
                                    {\normalfont\normalsize\bfseries}}
\renewcommand\subparagraph{\@startsection{subparagraph}{5}{\parindent}%
                                       {1.5ex \@plus.5ex \@minus .2ex}%
                                       {-1em}%
                                      {\normalfont\normalsize\bfseries}}
\makeatother

\newcommand{\msn}[1]{MR:\,\href{http://www.ams.org/mathscinet-getitem?mr=MR#1}{#1}}

\newcommand{\doi}[1]{doi:\,\href{http://dx.doi.org/#1}{#1}}

\theoremstyle{plain}
\newtheorem{thm}{Theorem}
\newtheorem{lem}[thm]{Lemma}
\newtheorem{conj}[thm]{Conjecture}

\newtheorem{prop}[thm]{Proposition}

\theoremstyle{definition}

\newcommand{\CEIL}[1]{\ensuremath{\protect\left\lceil#1\right\rceil}}

\newcommand{\ceil}[1]{\lceil{#1}\rceil}
\newcommand{\floor}[1]{\lfloor{#1}\rfloor}

\renewcommand{\geq}{\geqslant}
\renewcommand{\leq}{\leqslant}

\hypersetup{
    bookmarks=true,         
    unicode=false,          
    pdftoolbar=true,        
    pdfmenubar=true,        
    pdffitwindow=true,      
    pdftitle={My title},    
    pdfauthor={Author},     
    pdfsubject={Subject},   
    pdfnewwindow=true,      
    pdfkeywords={keywords}, 
    colorlinks=true,       
    linkcolor=blue,          
    citecolor=blue,        
    filecolor=blue,      
    urlcolor=blue           
}

\begin{document}

\title{Digraph 3-Colouring}
\author{}

{\Large\bfseries\boldmath\scshape Majority Colourings of Digraphs}

\DateFootnote

{\large 
Stephan Kreutzer\,\footnotemark[2] \quad 
Sang-il Oum\,\footnotemark[3]  \quad 
Paul Seymour\,\footnotemark[4]  \quad 
\\
 Dominic van der Zypen\quad  
 David~R.~Wood\,\footnotemark[5] 
}

\footnotetext[2]{Chair for Logic and Semantics, Technical University Berlin, Germany (\texttt{stephan.kreutzer@tu-berlin.de}). Research partly supported by DFG Emmy-Noether Grant Games and by the European Research Council (ERC) under the European Unions Horizon 2020 research and innovation programme (grant agreement No 648527).}

\footnotetext[3]{Department of Mathematical Sciences, KAIST, Daejeon, South Korea (\texttt{sangil@kaist.edu}).}

\footnotetext[4]{Department of Mathematics, Princeton University, New Jersey, U.S.A. (\texttt{pds@math.princeton.edu}).}

\footnotetext[5]{School of Mathematical Sciences, Monash University, Melbourne, Australia (\texttt{david.wood@monash.edu}).\\ 
Research supported by the Australian Research Council.}

\emph{Abstract.} We prove that every digraph has a vertex 4-colouring such that for each vertex $v$, at most half the out-neighbours of $v$ receive the same colour as $v$. We then obtain several results related to the conjecture obtained by replacing 4 by 3. 
\bigskip
\bigskip
\hrule

\renewcommand{\thefootnote}{\arabic{footnote}}

\section{Introduction}
\label{Intro}

A \emph{majority colouring} of a digraph is a function that assigns each vertex $v$ a colour, such that at most half the out-neighbours of $v$ receive the same colour as $v$. In other words, more than half the out-neighbours of $v$ receive a colour different from $v$ (hence the name `majority'). Whether every digraph has a majority colouring with a bounded number of colours was posed as an open problem on mathoverflow \citep{Zypen-mathoverflow}. In response, Ilya Bogdanov proved that a bounded number of colours suffice for tournaments. The following is our main result. 

\begin{thm}
\label{4Colouring}
Every digraph has a majority 4-colouring.
\end{thm}

\begin{proof}
Fix a vertex ordering. First, 2-colour the vertices left-to-right so that for each vertex $v$, at most half the out-neighbours of $v$ to the left of $v$ in the ordering receive the same colour as $v$. Second, 
2-colour the vertices right-to-left so that for each vertex $v$, at most half the out-neighbours of $v$ to the right of $v$ in the ordering receive the same colour as $v$. The product colouring is a majority 4-colouring.
\end{proof}

Note that this proof implicitly uses two facts: (1) every digraph has an edge-partition into two acyclic subgraphs, and (2) every acyclic digraph has a majority 2-colouring.

The following conjecture naturally arises: 

\begin{conj}
\label{Main}
Every digraph has a majority 3-colouring.
\end{conj}

This conjecture would be best possible. For example, a majority colouring of an odd directed cycle is proper (since each vertex has out-degree 1), and therefore three colours are necessary. There are examples with large outdegree as well. For odd $k\geq 1$ and prime $n\gg k$, let $G$ be the directed graph with $V(G)=\{v_0,\dots,v_{n-1}\}$ where $N_G^+(v_i)=\{v_{i+1},\dots,v_{i+k}\}$ and vertex indices are taken modulo $n$. Suppose that $G$ has a majority 2-colouring. If some sequence $v_i,v_{i+1},\dots,v_{i+k}$ contains more than $\frac{k+1}{2}$ vertices of one colour, say red, and $v_i$ is the leftmost red vertex in this sequence, then more than $\frac{k-1}{2}$ out-neighbours of $v_i$ are red, which is not allowed. Thus each sequence $v_i,v_{i+1},\dots,v_{i+k}$ contains exactly $\frac{k+1}{2}$ vertices of each colour. This implies that $v_i$ and $v_{i+k+1}$ receive the same colour, as otherwise the sequence $v_{i+1},\dots,v_{i+k+1}$ would contain more than $\frac{k+1}{2}$ vertices of the colour assigned to $v_{i+k+1}$. For all vertices $v_i$ and $v_j$, if $\ell=\frac{j-i}{k+1}$ in the finite field $\mathbb{Z}_n$, then $j=i+\ell(k+1)$ and $v_i,v_{i+(k+1)},v_{i+2(k+1)},\dots,v_{i+\ell(k+1)}=v_j$ all receive the same colour. Thus all the vertices receive the same colour, which is a contradiction. Hence the claimed 2-colouring does not exist. 

Note that being majority $c$-colourable is not closed under taking induced subgraphs. For example, let $G$ be the digraph with $V(G)=\{a,b,c,d\}$ and $E(G)=\{ab,bc,ca,cd\}$. Then $G$ has a majority 2-colouring: colour $a$ and $c$ by $1$ and colour $b$ and $d$ by $2$. But the subdigraph induced by $\{a,b,c\}$ is a directed 3-cycle, which has no majority 2-colouring.

The remainder of the paper takes a probabilistic approach to \cref{Main}, proving several results that provide evidence for  \cref{Main}. A probabilistic approach is reasonable, since in a random 3-colouring, one would expect that a third of the out-neighbours of each vertex $v$ receive the same colour as $v$. So one might hope that there is enough slack to prove that for \emph{every} vertex $v$, at most half the out-neighbours of $v$ receive the same colour as $v$. \cref{LargeOutdegree} proves \cref{Main} for digraphs with very large minimum outdegree (at least logarithmic in the number of vertcies), and then for digraphs with large minimum outdegree (at least a constant) and not extremely large maximum indegree. \cref{IndependentSets} shows that large minimum outdegree  (at least a constant) is sufficient to prove the existence of one of the colour classes in \cref{Main}. \cref{Generalisation} discusses multi-colour generalisations of \cref{Main}. 

Before proceeding, we mention some related topics in the literature:

\begin{itemize}

\item For undirected graphs, the situation is much simpler. \citet{Lovasz66} proved that for every undirected graph $G$ and integer $k\geq 1$, there is a $k$-colouring of $G$ such that every vertex $v$ has at most $\frac{1}{k}\deg(v)$ neighbours receiving the same colour as $v$. The proof is simple. Consider a $k$-colouring of $G$ that minimises the number of monochromatic edges. Suppose that some vertex $v$ coloured $i$ has greater than $\frac{1}{k} \deg(v)$ neighbours coloured $i$. Thus less than $\frac{k-1}{k}\deg(v)$ neighbours of $v$ are not coloured $i$, and less than $\frac{1}{k}\deg(v)$ neighbours of $v$ receive some colour $j\neq i$. Thus, if $v$ is recoloured $j$, then the number of monochromatic edges decreases. Hence no vertex $v$ has greater than $\frac{1}{k}\deg(v)$ neighbours with the same colour as $v$. 

\item \citet{Seymour74} considered digraph colourings such that every non-sink vertex receives a colour different from some outneighbour, and proved that a strongly-connected digraph $G$ admits a 2-colouring with this property if and only $G$ has an even directed cycle. The proof shows that every digraph has such a 3-colouring, which we repeat here: We may assume that $G$ is strongly connected. In particular, there are no sink vertices. Choose a maximal set $X$ of vertices such that $G[X]$ admits a 3-colouring where every vertex has a colour different from some outneighbour. Since any directed cycle admits such a colouring,  $X\neq\emptyset$. If $X \neq V(G)$, then choose an edge $uv$ entering $X$ and colour $u$ different from the colour of $v$, contradicting the maximality of $X$. So $X=V(G)$. (The same proof show two colours suffice if you start with an even cycle.)\ 

\item \citet{Alon96a,Alon06} posed the following problem: Is there a constant $c$ such that every digraph with minimum outdegree at least $c$ can be vertex-partitioned into two induced digraphs, one with minimum outdegree at least 2, and the other with minimum outdegree at least 1?

\item \citet{Wood-JCTB04}  proved the following edge-colouring variant of majority colourings: For every digraph $G$ and integer $k\geq 2$, there is a partition of $E(G)$ into $k$ acyclic subgraphs such that each vertex $v$ of $G$ has outdegree at most $\ceil{\frac{\deg^+(v)}{k-1}}$ in each subgraph. The bound $\ceil{\frac{\deg^+(v)}{k-1}}$  is best possible, since in each acyclic subgraph at least one vertex has outdegree 0.

\end{itemize}

\section{Large Outdegree}
\label{LargeOutdegree}

We now show that minimum outdegree at least logarithmic  in the number of vertices is sufficient to guarantee a majority 3-colouring. All logarithms are natural.

\begin{thm}
\label{LogarithmicOutdegree}
Every graph  $G$ with $n$ vertices and minimum outdegree $\delta> 72\log(3n)$ has a majority 3-colouring. Moreover, at most half the out-neighbours of each vertex receive the same colour. 
\end{thm}

\begin{proof}
Randomly and independently colour each vertex of $G$ with one of three colours $\{1, 2, 3\}$. Consider a vertex $v$ with out-degree $d_v$. Let $X(v,c)$ be the random variable that counts the number of out-neighbours of $v$ coloured $c$. Of course, $\mathbf{E}(X(v,c)) = d_v/3$. Let $A(v,c)$ be the event that $X(v,c) > d_v/2$. Note that  $X(v,c)$ is determined by $d_v$ independent trials and changing the outcome of any one trial changes $X(v,c)$ by at most 1. By the simple concentration bound\footnote{The simple concentration bound says that if $X$ is a random variable determined by $d$ independent trials, such that changing the outcome of any one trial can affect $X$ by at most $c$, then $\textbf{P}(X >\textbf{E}(X) + t) \leq \exp(-t^2/2c^2d)$; see \citep[Chapter 10]{MR02}. With $\textbf{E}(X_v)=d_v/3$ and $t=d_v/6$ and $c=1$ we obtain the desired upper bound on $\textbf{P}(X_v > d_v/2)$.},
\begin{align*}
\mathbf{P}(A(v, c)) 
 \leq \exp(-(d_v/6)^2/2d_v)
 = \exp(-d_v/72)
\leq \exp(-\delta/72).
\end{align*}
The expected number of events $A(v,c)$ that hold is $$\sum_{v\in V(G)}\sum_{c\in \{1,2,3\}}\mathbf{P}(A(v, c)) \leq 3n  \exp(-\delta/72)<1,$$
where the last inequality holds since $\delta> 72\log(3n)$. Thus there exists colour choices such that no event  $A(v,c)$ holds. That is, a majority 3-colouring exists. 
\end{proof}

The following result shows that large outdegree (at least a constant) and not extremely large indegree is sufficient to guarantee a majority 3-colouring.

\begin{thm}
\label{LocalLemmaColouring}
Every digraph with minimum out-degree  $\delta\geq 1200$ and maximum in-degree at most $\exp(\delta/72)/12\delta$ has a majority 3-colouring. Moreover, at most half the out-neighbours of each vertex receive the same colour. 
\end{thm}

\begin{proof}
We assume $\delta\geq 1200$, as otherwise the minimum out-degree $\delta$ is greater than the maximum in-degree $\exp(\delta/72)/12\delta$, which does not make sense.
 
 We use the following weighted version of the Local Lemma \citep{EL75,MR02}: Let $\mathcal{A} := \{A_1, \dots , A_n\}$ be a set of `bad' events, such that each $A_i$ is mutually independent of $\mathcal{A}\setminus (D_i \cup \{A_i\})$, for some subset $D_i \subseteq  A$. Assume there are numbers $t_1,\dots,t_n \geq 1$ and a real number $p \in [0, \frac14]$ such that for $1 \leq  i \leq  n$,
$$(a)\; \mathbf{P}(A_i) \leq  p^{t_i} \quad \text{and} \quad  (b)\; \sum_{A_j\in \mathcal{D}_i} (2p)^{t_j} \leq  t_i/2.$$ 
Then with positive probability no event $A_i$ occurs.

Define $p := \exp(-\delta /72)$. Since $\delta \geq1200$ we have $p\in [0, \frac14]$. Randomly and independently colour each vertex of $G$ with one of three colours $\{1, 2, 3\}$. Consider a vertex $v$ with out-degree $d_v$. Let $X(v,c)$ be the random variable that counts the number of out-neighbours of $v$ coloured $c$. Of course, $\mathbf{E}(X(v,c)) = d_v/3$. Let $A(v,c)$ be the event that $X(v,c) > d_v/2$. Let $\mathcal{A} := \{A(v,c) : v \in  V(G),c \in  \{1,2,3\}\}$  be our set of events. Let $t(v,c) := t_v := d_v/\delta$  be the associated weight. Then $t_v \geq  1$. It suffices to prove that conditions (a) and (b) hold. 

Note that  $X(v,c)$ is determined by $d_v$ independent trials and changing the outcome of any one trial changes $X(v,c)$ by at most 1. By the simple concentration bound,
\begin{align*}
\mathbf{P}(A(v, c)) 
 \leq \exp(-(d_v/6)^2/2d_v)
 = \exp(-d_v/72)
 = \exp(-\delta t_v/72)
 = p^{t_v} .
\end{align*}
Thus condition (a) is satisfied.
For each event $A(v,c)$ let $D(v,c)$ be the set of all events $A(w,c') \in \mathcal{A}$ such that $v$
and $w$ have a common out-neighbour. Then $A(v, c)$ is mutually independent of $\mathcal{A} \setminus (D(v, c) \cup  \{A(v, c)\})$. Since $t_w \geq  1$,
\begin{align*}
\sum_{A(w,c')\in D(v,c)}\!\!\!\!\!\!\!\!  (2p)^{t_w} 
 \leq  \sum_{A(w,c')\in D(v,c)}\!\!\!\!\!\!\!\! (2p)^1 
 =  2p |D(v, c)|.
 \end{align*}
Since each out-neighbour of $v$ has in-degree at most $\exp(\delta /72)/12\delta$, we have $|D(v,c)| \leq  d_v\exp(\delta /72)/4\delta$ and 
\begin{align*}
\sum_{A(w,c')\in D(v,c)}  (2p)^{t_w} 
 \leq p d_v \exp(\delta/72)/2\delta
 =  \exp(-\delta/72) t_v \exp(\delta/72)/2
 =  t_v/2.
 \end{align*}
Thus condition (b) is satisfied. By the local lemma, with positive probability, no event $A(v,c)$ occurs. That is, a majority 3-colouring exists. 
\end{proof}
 
Note that the conclusion in \cref{LogarithmicOutdegree} and  \cref{LocalLemmaColouring} is stronger than in Conjecture~\ref{Main}. We now show that such a conclusion is impossible (without some extra degree assumption). 

\begin{lem}
For all integers $k$ and $\delta$, there are infinitely many digraphs $G$ with minimum outdegree $\delta$, such that for every vertex $k$-colouring of $G$, there is a vertex $v$ such that all the out-neighbours of $v$ receive the same colour. 
\end{lem}

\begin{proof}
Start with a digraph $G_0$ with  at least $k\delta$ vertices and minimum outdegree $\delta$. For each set $S$ of $\delta$ vertices in $G_0$, add a new vertex with out-neighbourhood $S$. Let $G$ be the digraph obtained. In every $k$-colouring of $G$, at least $\delta$ vertices in $G_0$ receive the same colour, which implies that for some vertex 
$v\in V(G)\setminus V(G_0)$, all the out-neighbours of $v$ receive the same colour. 
\end{proof}

\section{Stable Sets}
\label{IndependentSets}

A set $T$ of vertices in a digraph $G$ is a \emph{stable set} if for each vertex $v\in T$, at most half the out-neighbours of $v$ are also in $T$. A majority colouring is a partition into stable sets. Of course, if a digraph has a majority 3-colouring, then it contains a stable set with at least one third of the vertices. The next lemma provides a sufficient condition for the existence of such a set.

\begin{thm}
\label{IndSet}
Every digraph $G$ with $n$ vertices and minimum outdegree at least $22$ has a stable set with at least $\frac{n}{3}$ vertices.
\end{thm}

\cref{IndSet} is proved via the following more general lemma.

\begin{lem}
\label{GenIndSet}
For $0<\alpha<p<\beta<1$, every digraph $G$ with minimum outdegree at least 
$$\delta:=\CEIL{ \frac{(\beta+p) \log\left(\frac{p}{p-\alpha}\right)}{(\beta-p)^2 } }$$ 
contains a set $T$ of at least $\alpha n$ vertices, such that $|N^+_G(v)\cap T| \leq \beta |N^+_G(v)|$  for every vertex $v\in T$. 
\end{lem}

\begin{proof}
Let $d_v:=|N^+_G(v)|$ be the outdegree of each vertex $v$ of $G$. Initialise $S:=\emptyset$. For each vertex $v$ of $G$, add $v$ to $S$ independently and randomly with probability $p$. Let $X_v:=|N^+_G(v)\cap S|$. 
Note that $X_v\sim\text{Bin}(d_v,p)$ and 
\begin{equation}
\label{PXv}
\mathbf{P}(X_v > \beta d_v) = \sum_{k\geq\floor{\beta d_v}+1}^{d_v} \binom{d_v}{k} p^k (1-p)^{d_v-k}.
\end{equation}
By the Chernoff bound\footnote{The Chernoff bound implies that if $X\sim\text{Bin}(d,p)$ then $\textbf{P}(X\geq (1+\epsilon)pd) \leq \exp( - \frac{\epsilon^2}{2+\epsilon}\, pd)$ for $\epsilon\geq 0$. With $\epsilon=\frac{\beta}{p}-1$ we have $\textbf{P}(X> \beta d) \leq \exp( - \frac{(\beta-p)^2}{p+\beta}\, d)$.},
\begin{equation}
\label{PXvdesired}
\mathbf{P}(X_v > \beta d_v) \leq \exp\left( - \frac{(\beta-p)^2}{\beta+p} d_v  \right) 
\leq \exp\left( - \frac{(\beta-p)^2}{\beta+p} \delta  \right)
\leq \frac{p-\alpha}{p}. 
\end{equation}
where the last inequality follows from the  definition of $\delta$.  
Let $B:=\{v\in S:X_v > \beta d_v\}$. Then 
\begin{align*}
\mathbf{E}(|B|) 
 = \sum_{v\in V(G)} \mathbf{P}( v\in S \text{ and }X_v > \beta d_v) .
\end{align*}
Since the events $v\in S$ and $X_v > \beta d_v$ are independent,
\begin{align*}
\mathbf{E}(|B|) 
 = \sum_{v\in V(G)} \mathbf{P}( v\in S ) \,\mathbf{P}(X_v > \beta d_v) 
 = p \sum_{v\in V(G)} \mathbf{P}(X_v > \beta d_v) 
 \leq (p-\alpha)n.
\end{align*}
Let $T:=S\setminus B$. Thus $|N^+_G(v)\cap T| \leq \beta d_v$ for each vertex $v\in T$, as desired.
By the linearity of expectation, 
$$\mathbf{E}(|T|)=\mathbf{E}(|S|)-\mathbf{E}(|B|)=pn-\mathbf{E}(|B|) \geq \alpha n.$$
Thus there exists the desired set $T$. 
\end{proof}

\begin{proof}[Proof of \cref{IndSet}] The proof follows that of \cref{GenIndSet} with one change. Let $\alpha:=\frac13$ and $\beta:=\frac12$ and $p:=0.38$. Then $\delta=129$. If $22\leq d_v\leq 128$ then direct calculation of the formula in \eqref{PXv} verifies that $\mathbf{P}(X_v > \beta d_v) \leq  \frac{p-\alpha}{p}$, as in \eqref{PXvdesired}. For $d_v\geq 129$ the Chernoff bound proves \eqref{PXvdesired}. The rest of the proof is the same as in \cref{GenIndSet}.
\end{proof}


Note the following corollary of \cref{GenIndSet} obtained with $\alpha=\frac12-\epsilon$ and $p=\frac12-\frac{\epsilon}{2}$. This says that graphs with large minimum outdegree have a stable set with close to half the vertices.

\begin{prop}
For $0<\epsilon<\frac12$, every $n$-vertex digraph $G$ with minimum outdegree at least 
$2\epsilon^{-2}(2-\epsilon) \log(\tfrac{1-\epsilon}{\epsilon})$ 
contains a stable set of at least $(\frac12-\epsilon)n$ vertices.
\end{prop}


%

%
%

\section{Multi-Colour Generalisation}
\label{Generalisation}

The following natural generalisation of Conjecture~\ref{Main} arises. 


\begin{conj}
\label{kColours}
For $k\geq 2$, every digraph has a vertex $(k+1)$-colouring such that for each vertex $v$, at most $\frac{1}{k}\deg^+(v)$ out-neighbours of $v$ receive the same colour as $v$. 
\end{conj}

The proof of \cref{4Colouring} generalises to give an upper bound  of $k^2$ on the number of colours in \cref{kColours}. It is open whether the number of colours is $O(k)$. This conjecture would be best possible, as shown by the following example. Let $G$ be the $k$-th power of an $n$-cycle, with arcs oriented clockwise, where $n\geq 2k+3$ and $n\not\equiv 0\pmod{k+1}$. Each vertex has outdegree $k$. Say $G$ has a vertex $(k+1)$-colouring such that for each vertex $v$, at most $\epsilon k$ out-neighbours of $v$ receive the same colour as $v$. If $\epsilon k<1$ then the  underlying undirected graph of $G$ is properly coloured, which is only possible if $n\equiv 0\pmod{k+1}$. Hence $\epsilon\geq\frac{1}{k}$. 

%

%

\cref{GenIndSet} with $\alpha=\frac{1}{k}-\epsilon$ and $\beta=\frac{1}{k}$ and $p=\frac{1}{k}-\frac{\epsilon}{2}$ implies the following `stable set' version of  \cref{kColours} for digraphs with large minimum outdegree. 

\begin{prop}
\label{kIndSet}
For $k\geq 2$ and $\epsilon\in(0,\frac{1}{k})$, every $n$-vertex digraph $G$ with minimum outdegree at least 
$2\epsilon^{-2} (\tfrac{4}{k}-\epsilon) \log\left(\tfrac{2}{\epsilon k}-1\right)$
contains a set $T$ of at least $(\frac{1}{k}-\epsilon)n$ vertices, such that  for every vertex $v\in T$, at most $\frac{1}{k}\deg^+(v)$ out-neighbours of $v$ are also in $T$.
\end{prop}

\section{Open Problems}

In addition to resolving \cref{Main}, the following open problems arise from this paper:

\begin{enumerate}

\item Is there a constant $\beta<1$ for which every digraph has a 3-colouring, such that for every vertex $v$, at most $\beta\deg^+(v)$ out-neighbours receive the same colour as $v$?

\item Does every tournament have a majority 3-colouring?

\item Does every Eulerian digraph have a majority 3-colouring? Note that for an Eulerian digraph $G$, if each vertex $v$ has in-degree and out-degree $\deg(v)$, then by the result for undirected graphs mentioned in \cref{Intro}, the underlying undirected graph of $G$ has a 4-colouring such that each vertex $v$ has at most $\frac{1}{2}\deg(v)$ in- or- out-neighbours with the same colour as $v$. In particular, $G$ has a majority 4-colouring. By an analogous argument every Eulerian digraph has a 3-colouring such that each vertex $v$ has at most $\frac{2}{3}\deg(v)$ in- or- out-neighbours with the same colour as $v$, thus proving a special case of the first question above. 

\item Does every digraph in which every vertex has in-degree and out-degree $k$ have a majority 3-colouring? A variant of \cref{LocalLemmaColouring} proves this result for $k\geq 144$. 

\item Is there a characterisation of digraphs that have a majority 2-colouring (or a polynomial time algorithm to recognise such digraphs)?


\item Does every digraph have a $O(k)$-colouring such that for each vertex $v$, at most $\frac{1}{k}\deg^+(v)$ out-neighbours receive the same colour as $v$ (for all $k\geq 2$)?

\item A digraph $G$ is \emph{majority $c$-choosable} if for every function $L:V(G)\rightarrow\mathbb{Z}$ with $|L(v)|\geq c$ for each vertex $v\in V(G)$, there is a majority colouring of $G$ with each vertex $v$ coloured from $L(v)$. Is every digraph majority $c$-choosable for some constant $c$? The proof of \cref{4Colouring} shows that acyclic digraphs are majority 2-choosable, and obviously \cref{LogarithmicOutdegree} and \cref{LocalLemmaColouring} extend to the setting of choosability. 

\item Consider the following fractional setting. Let $S(G)$ be the set of all stable sets of a digraph $G$. Let $S(G,v)$ be the set of all stable sets containing $v$. A \emph{fractional majority colouring} is a function that assigns each stable set $T\in S(G)$ a weight $x_T\geq 0$ such that  $\sum_{T\in S(G,v)}x_T\geq 1$ for each vertex $v$ of $G$. What is the minimum number $k$ such that every digraph $G$ has a fractional majority colouring with total weight $\sum_{T\in S(G)}x_T\leq k$? Perhaps it is less than 3.

\end{enumerate}

\subsection*{Acknowledgements} This research was initiated at the Workshop on Graph Theory at  Bellairs Research Institute  (March 25 -- April 1, 2016).


\end{document}